\documentclass{article}
\usepackage{graphicx}
\usepackage{amsmath,graphicx,amsfonts,amsthm,amssymb,setspace,mathtools,
	float,tikz,caption,fancyhdr,hyperref}
\usetikzlibrary{calc}
\graphicspath{{figures/}}
\author{K\^ohei Sasaya
	\thanks{Kyoto University, Kyoto 606-8502, Japan. JSPS Research Fellow (DC1). E-mail: \texttt{ksasaya@kurims.kyoto-u.ac.jp}}
}

\title{Systems of Dyadic Cubes of Complete, Doubling, Uniformly Perfect Metric Spaces without Detours}
\makeatletter

\@addtoreset{equation}{section}
\makeatother

\newtheorem{lem}{Lemma}[section]
\newtheorem{prop}[lem]{Proposition}
\newtheorem{thm}[lem]{Theorem}

\theoremstyle{definition}
\newtheorem{defi}[lem]{Definition}

\theoremstyle{remark}
\newtheorem*{rem}{Remark}

\newcommand{\ard}{\dim_\mathrm{AR}}
\newcommand{\diam}{\mathrm{diam}}

\newcommand{\ol}[1]{\overline{#1}}
\newcommand{\ul}[1]{\underline{#1}}
\newcommand{\Mc}[1]{\mathcal{#1}}
\newcommand{\Mb}[1]{\mathbb{#1}}

\renewcommand{\Cup}{\bigcup}
\renewcommand{\Cap}{\bigcap}

\newcommand{\sCup}{\bigsqcup}

\newcommand{\nai}{\mathrm{int}}

\newcommand{\doi}[1]{\url{https://doi.org/#1}}

\begin{document}
\maketitle
\begin{abstract}
 Systems of dyadic cubes are the basic tools of harmonic analysis and geometry, and this notion had been extended to general metric spaces. In this paper, we construct systems of dyadic cubes of complete, doubling, uniformly perfect metric spaces, such that for any two points in the metric space, there exists a chain of three cubes whose diameters are comparable to the distance of the points. We also give an application of our construction to previous research of potential analysis and geometry of metric spaces.
\end{abstract}	
\section{Introduction}
The standard system of dyadic cubes of $\Mb{R}^d$, 
\begin{equation}
	S_k=\Bigl\{\prod_{i=1}^d\bigl[\frac{a_i}{2^k},\frac{a_i+1}{2^k}\bigr]\mid a_i\in\Mb{Z}\ (1\le i\le d)\Bigr\} \quad(k\in\Mb{Z}) \label{SD}
\end{equation}
is a basic tool of analysis on the Euclidean spaces. Constructing the counterparts of dyadic cubes of general metric spaces were started by David \cite{Da88,Da91} and Christ \cite{Chr} for metric measure spaces. Hyt\"onen and Kairema \cite{HK12} and Ka\"enm\"aki, Rajala and Suomala \cite{KRS} extended these results for metric spaces without measures. Systems of generalized dyadic cubes were used for various studies of harmonic analysis (e.g. \cite{AH,BFP}) and analysis on metric spaces (e.g. \cite{CJKS,kig}).\\
From the viewpoint of discrete approximation of a metric space, it is important whether the structure of a system of dyadic cubes is comparable to that of the original metric space or not. For example, \eqref{SD} satisfies
\begin{multline*}
	\text{for any $x,y\in\Mb{R}^d$ and $k\in\Mb{Z}$ with $|x-y|_{\Mb{R}^d}\le2^{-k},$ there exist}\\
	\text{$Q_x,Q_y\in S_k$ such that $x\in Q_x, y\in Q_y$ and $Q_x\cap Q_y\ne\emptyset.$ }
\end{multline*}
However, a system of dyadic cubes of a metric space may not satisfy such a condition, for instance, two close points may not have any short chain of dyadic cubes (see Figure \ref{apart}). 
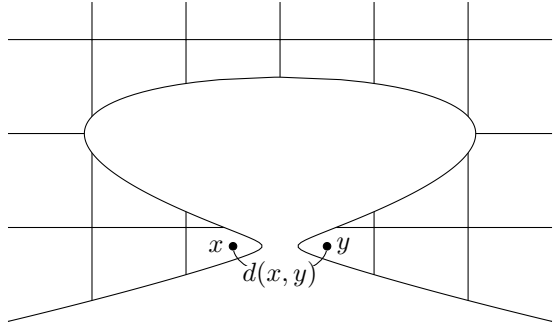
\begin{figure}[tb]
	\centering
	\begin{tikzpicture}[scale=0.25]
		\draw plot[samples=100,domain=-5:8]({-sqrt((-(\x)+8)*((\x)^2+2*(\x)+1.1))},{(\x)})--plot[samples=100,domain=-8:5]({sqrt(((\x)+8)*((\x)^2-2*(\x)+1.1))},{-(\x)});
		\begin{scope}
			\clip
			plot[samples=100,domain=-5:8]({-sqrt((-(\x)+8)*((\x)^2+2*(\x)+1.1))},{(\x)})--plot[samples=100,domain=-8:5]({sqrt(((\x)+8)*((\x)^2-2*(\x)+1.1))},{-(\x)})--({sqrt(209.3)},12)--({-sqrt(209.3)},12)--cycle;
			\draw[step=5](-15,-5) grid (15,12);
		\end{scope}
			\coordinate[label=left :$x$] (a1) at (-2.5,-1) ;
			\filldraw (a1) circle [x radius=0.2, y radius=0.2];
			\coordinate[label=right:$y$] (a2) at (2.5,-1) ;
			\filldraw (a2) circle [x radius=0.2, y radius=0.2];
			\node (a3) at (0,-2.5) {$d(x,y)$} ;
			\begin{scope}
				\clip (-2.5,-1.3)--(-2.5,-2)--(2.5,-2)--(2.5,-1.3)--cycle;
				\draw (0,-1) circle [x radius=2.5, y radius=1.5];
			\end{scope}
	\end{tikzpicture}
	\caption{Two close points without short chains}
	\label{apart}
\end{figure}
Hence it is a natural question when a system of dyadic cubes $S_k\ (k\in\Mb{Z})$ of a metric space $(X,d)$, satisfying the following conditions
for some $C,C',r>0$ and $M\in\Mb{N}$ exist:
\begin{itemize}
	\item $C^{-1}r^k\le\diam (Q,d)\le Cr^k$ for any $Q\in S_k$
	\item 
		for any $x,y\in X$ and $k\in\Mb{Z}$ with $d(x,y)\le C'r^k,$ there exist
		$\{Q_i\}_{i=0}^M\subset S_k$ such that $x\in Q_0,y\in Q_M$ and $Q_{i-1}\cap Q_i\ne\emptyset\ (1\le i\le M).$

\end{itemize}

The aim of this paper is to give an answer to the question in a constructive way, under certain conditions. \\

Let $(X,d)$ be a metric space, $B(x,r)$ be an open ball of radius $r$, and $\ol{A}$ (resp. $\nai A$) be the closure (resp. interior) of a set $A\subset X.$
\begin{defi}[Doubling]
	$(X,d)$ is called doubling if there exists $N\in\Mb{N}$ such that for any $x\in X$ and $r>0,$ there exist $\{x_i\}_{i=1}^N\subset X$ with $B(x,2r) \subset \Cup_{i=0}^N B(x_i,r).$
\end{defi}
\begin{defi}[Uniformly perfect]\label{up}
 $(X,d)$ is called uniformly perfect if there exists $\gamma>1$ such that $B(x,\gamma r)\setminus B(x,r)\ne\emptyset$ whenever $B(x,r)\ne X.$
\end{defi}
Throughout this paper, we assume $(X,d)$ is complete, doubling and uniformly perfect. Our main result is the following.

\begin{thm}\label{Main}
	Let $C_*,c_*$ be constants with $0<c_*<C_*<\infty.$ Then there exist $C_1, C_2, C_3, r_0>0$ such that if a set of points $\cup_{k\in\Mb{Z}}\{x_\omega\}_{\omega\in\Omega_k}$ with some $\Omega_k\ (k\in\Mb{Z})$ satisfies 
	\begin{align}
		d(x_\omega,x_\lambda)\ge c_*r^k &\quad\text{for any $\omega,\lambda$ with $\omega\ne\lambda$} \label{c}\\
		\min_{\omega\in\Omega_k}d (y, x_\omega)<C_*r^k &\quad\text{for any $y\in X$} \label{C}
	\end{align}
for some $r\in(0,r_0],$ then there exist $\{Q_\omega\subset X\mid k\in\Mb{Z}, \omega\in\Omega_k\}$ such that
\begin{spacing}{1.3}
	\vspace{-25pt}
	\begin{align}
		\bullet\ &\Cup_{\omega\in\Omega_k} Q_\omega=X\ \text{for any }k\in\Mb{Z}, \tag{D1}\label{q1} \\
		\bullet\ &\nai Q_\omega=\nai(\ol{Q_\omega}), \quad \ol{\nai Q_\omega}=\ol{Q_\omega}, \tag{D2}\label{q2} \\
		\bullet\ &\text{if $l\ge k,$ then either $Q_\omega\subset Q_\lambda$ or $Q_\omega\cap Q_\lambda=\emptyset$ holds}\notag \\
			&\text{for any $\omega\in\Omega_l,\ \lambda\in\Omega_k,$}\tag{D3}\label{q3}\\
		\bullet\ &B(x_\omega, C_1r^k)\subset Q_\omega\subset B(x_\omega, C_2r^k), \tag{D4}\label{q4}\\
		\bullet\ &\text{if $y,z\in X$ satisfy $d(y,z)\le C_3r^k,$ then there exist $\omega_0,\omega_1,\omega_2\in\Omega_k$}\notag\\
			&\text{such that $\ol{Q_{\omega_i}}\cap\ol{Q_{\omega_{i+1}}}\ne\emptyset\ (i=0,1)$ and $y\in \ol{Q_{\omega_0}}, \ z\in \ol{Q_{\omega_2}}.$
			} \tag{D5}\label{q5}
	\end{align}
\end{spacing}\vspace{-10pt}
\end{thm}
This theorem is based on the results and proofs of \cite{Chr} and \cite{HK12}, and our main contribution is to develop the proof in order to satisfy the additional condition \eqref{q5}. In \cite{HK12}, Hyt\"onen and Kairema proved that there exist $C>0,\ N\in\Mb{N},\ \{\{\Omega_k^{(i)}\}_{k\in\Mb{Z}}\}_{1\le i\le N}$ and $\{Q_\omega\mid \omega\in\Cup_{k\in\Mb{Z}}\Omega_k^{(i)}\}_{1\le i\le N},$ which satisfy \eqref{q1} to \eqref{q4} for each $i$ and 
\begin{multline}\label{var}
\text{for any $x,y\in X$ and $k\in\Mb{Z}$ with $d(x,y)\le Cr^k,$}\\
\text{there exist $i\le N$ and $\omega\in\Omega_k^{(i)}$ such that $\{x,y\}\subset Q_\omega$ }
\end{multline}
only using the doubling condition, but $i$ in \eqref{var} may be different for each point. \\

In order to state an example of previous results which use a system of dyadic cubes with \eqref{q5}, we prepare some conditions.
\begin{defi}[Ahlfors regular]
	Let $(Y,\rho)$ be a metric space and $\mu$ be a Borel measure on it. We say $\mu$ is $\alpha$-Ahlfors regular with respect to $(Y,\rho)$ if there exists $C>0$ such that 
	\[C^{-1}r^\alpha\le \mu(B_\rho(x,r)) \le Cr^\alpha\hspace{5pt}\text{ for any }x\in Y\text{ and }r_x\le r\le\diam(Y,\rho)\]  
	where $r_x=r_{x,\rho}=\inf_{y\in Y\setminus\{x\}}\rho(x,y).$ $(Y,\rho)$ is called $\alpha$-Ahlfors regular if there exists a Borel measure $\mu$ such that $\mu$ is $\alpha$-Ahlfors regular with respect to $(Y,\rho).$
\end{defi}
Ahlfors regularity is sometimes used in studies of harmonic analysis via dyadic cubes (e.g. \cite{GL}). We also note that if $(Y,\rho)$ is an $\alpha$-Ahlfors regular space without isolated points, then the Hausdorff dimension of $(Y,\rho)$ coincides to $\alpha,$ and $(Y,\rho)$ is uniformly perfect.
\begin{defi}[Quasisymmetry, in the sense of \cite{kig}]
		Let $Y$ be a set and $\rho,\delta$ be distances on $Y.$ We say that $\rho$ is quasisymmetric to $\delta$ if there exists a homeomorphism $\theta:[0,\infty)\to[0,\infty)$ such that for any $x,y,z\in Y$ with $x\ne z,$
	\[ (\delta(x,y)/\delta(x,z))\le\theta(\rho(x,y)/\rho(x,z) ).\]
\end{defi}
\begin{defi}[Ahlfors regular conformal dimension]\label{defARC}
	The Ahlfors regular conformal dimension (ARC dimension in short) of $(Y,\rho)$ is defined by 
	\begin{multline*}
		\ard(Y,\rho)=\inf\{\alpha\mid \text{ there exists a metric $\delta$ on $Y$ such that}\\
		\text{$(Y,\delta)$ is $\alpha$-Ahlfors regular and $\delta$ is quasisymmetric to }\rho\},  
	\end{multline*}
	where $\inf\emptyset=\infty.$
\end{defi}
Quasisymmetry was introduced by Tukia and V\"ais\"ala in~\cite{TV} as a condition of embedding maps between metric spaces, and it is a generalization of quasiconformal mappings on the complex plane.
ARC dimension was implicitly introduced by Bourdon and Pajot \cite{BP}, and named by Bonk and Kleiner \cite{BK}. In \cite{BK}, the ARC dimension is related to Cannon's conjecture, which is famous in the study of hyperbolic groups: it claims that for any hyperbolic group $G$ whose boundary is homeomorphic to $\Mb{S}^2$, there exists a geometric action of $G$ on the hyperbolic space $\Mb{H}^3.$ For compact metric spaces, in \cite{kig} Kigami showed that if the counterpart of a system of dyadic cubes satisfying \eqref{q1} to \eqref{q5}, called ``partition satisfying basic framework'' exists, then $\ard(Y,\rho)$ is characterized by the $p$-energies ($p>0$) on the graphs defined by the system of dyadic cubes (see Theorem \ref{thar}). In Section 3, we introduce the detailed conditions and results of \cite{kig} and its extension to $\sigma$-compact metric spaces and graphs, \cite{sas}. We also state the application of our main result to them in that section.
%xの下つきのペアのかっこは適宜省略
%\min\emptyset=\infty, \vee, \wedge, \#
%\subsection*{Notations}
\section{Remarks and proof of main theorem}
Before the proof of our main theorem, we give some remarks on Theorem \ref{Main}. We first note that the existence of $\{x_\omega\}_{\omega\in\Omega_k}$ with \eqref{c} and \eqref{C} is easily shown as noted in \cite[Subsection 2.12]{HK12}.

\begin{lem}\label{net}
	Let $A=\{x_\omega\}_{\omega\in\Omega^*_k}$ be points with \eqref{c}, then there exists $B=\{x_\omega\}_{\omega\in\Omega_k}$ with \eqref{c}, \eqref{C} and $A\subset B.$
\end{lem}
\begin{proof}
	It is shown by the maximal argument with Zorn's lemma.
\end{proof}
\begin{rem}
	We can avoid Zorn's lemma if we know $(X,d)$ is separable. A doubling metric space is separable, but this fact is shown by the existence of points satisfying \eqref{c} and \eqref{C} in general. 
\end{rem}
We also note that the results in \cite{Chr} and \cite{HK12} were stated for quasimetrics, that is, for $(Y,\rho)$ satisfying the axioms of a metric spaces but replacing  the triangle inequality with the following condition:
\[\text{there exists }A\ge1\text{ such that }\rho(x,z)\le A(\rho(x,y)+\rho(y,z))\text{ for any }x,y,z\in Y.\]
It is known that if $\rho$ is a quasimetric, then there exists $c,\beta\in(0,1)$ and a metric $\rho'$ such that $c\rho'(x,y)\le \rho^\beta(x,y)\le c^{-1}\rho'(x,y)$ for any $x,y\in Y$ (see \cite[Proposition 14.5]{Hei}), so we can apply Theorem \ref{Main} to quasimetric spaces with this fact.

\subsection{Proof of Theorem \ref{Main}}
For simplicity of notation, we write $\max\{a,b\}=a\vee b,$ $\min\{a,b\}=a\wedge b$ and $\Omega=\sCup_{k\in\Mb{Z}}\Omega_k.$ We also write $\# A$ for the number of vertices of a set $A.$
We first prove Theorem \ref{Main} under strong conditions.
Since $(X,d)$ is doubling, we may assume $\Omega_k$ in the assumption of Theorem \ref{Main} is countable, so we can write $\Omega_k=\{(k,n)\mid n\in\Mb{N}, n<N_k\}$ for some $N_k\in \Mb{N}\cup\{\infty\}.$
\begin{prop}\label{TtoM}
Let $\alpha_1,\alpha_2,\alpha_3>0$ and $r_0\in(0,1)$ satisfy
\begin{equation}
	\frac{r_0}{1-r_0}\alpha_1<\alpha_2\wedge(\alpha_3-C_*). \label{cond1}
\end{equation}
 We assume that for any $r\in(0,r_0]$ and $A=\cup_{k\in\Mb{Z}}\{x_\omega\}_{\omega\in\Omega_k}$ with \eqref{c} and \eqref{C}, there exists a map $\pi=\pi_A: \Omega\to\Omega$ such that for any $k\in\Mb{Z},$
\begin{align}
	\bullet\ & \pi(\Omega_{k+1})=\Omega_k \text{ and } \sup_{\omega\in\Omega}\#\pi^{-1}(\omega)<\infty, \tag{T1}\label{T1}\\
	\bullet\ & d(x_{k+1,m},x_{\pi(k+1,m)})<\alpha_1 r^k, \tag{T2}\label{T2}\\
	\bullet\ &\text{if } d(x_{k,n},x_{k+1,m})<\alpha_2r^k\text{ then } \pi(k+1,m)=(k,n), \tag{T3}\label{T3}\\
	\bullet\ &\text{if }B(x_{k,n_1}, \alpha_3r^k)\cap B(x_{k,n_2},\alpha_3r^k)\ne\emptyset, \text{ then there exist }m_1, m_2\in\Mb{N}\notag\\
	&\text{such that }B(x_{k+1,m_1}, \alpha_3r^{k+1})\cap B(x_{k+1,m_2},\alpha_3r^{k+1})\ne\emptyset \text{ and }\notag \\
	&\pi(k+1,m_i)=(k,n_i)\ (i=1,2),\tag{T4}\label{T4} \\
	\bullet\ & \text{if }d(x_{k+1,m_1},x_{\pi(k+1,m_1)})\ge C_*r^k\text{ and }d(x_{k+1,m_1},x_{k,n})<\alpha_3r^k, \text{ then}\notag\\
	&\text{there exist } m_2, m_3 \text{ such that } B(x_{k+1,m_2}, \alpha_3r^{k+1})\cap B(x_{k+1,m_3},\alpha_3r^{k+1})\ne\emptyset, \notag \\
	&\pi(k+1,m_1)=\pi(k+1,m_2)\text{ and }(k,n)=\pi{(k+1,m_3)}. \tag{T5}\label{T5}
\end{align}
Then, Theorem \ref{Main} holds.
\end{prop} 
In order to prove this proposition, we prepare some technical lemmas. We let $K_\omega=\ol{\Cup_{l\ge0}\{x_\lambda\mid\lambda\in\pi^{-l}(\omega)\}}$ for any $\omega\in\Omega.$
\begin{lem}\label{hage}
	Let $\alpha_4=\frac{\alpha_1}{1-r_0}$ and $\alpha_5=\alpha_2-\frac{\alpha_1r_0}{1-r_0}.$
	\begin{enumerate}
	\item $K_{k,n}\subset \ol{B(x_{k,n},\alpha_4r^k)}$ for any $(k,n).$
	\item $\Cup_{\omega\in\Omega_k}K_\omega=X$ for any $k.$
	\item $B(x_{k,n},\alpha_5r^k)\subset K_{k,n}$ for any $(k,n).$
\end{enumerate}
\end{lem}
\begin{proof}
	\begin{enumerate}
		\item If $\lambda\in\pi^{-l}(\omega)$ for $l\ge1,$ then $d(x_\omega, x_\lambda)\le\sum_{j=0}^{l-1}d(x_{\pi^j(\lambda)},x_{\pi^{j+1}(\lambda)}) < \frac{\alpha_1}{1-r}r^k\le\alpha_4r^k$ by \eqref{T1} and \eqref{T2}. This suffices to show (1).
		\item Let $y\in\ol{\Cup_{\omega\in\Omega_k}K_\omega}.$ Since $(X,d)$ is doubling and $\{x_\omega\}_{\omega\in\Omega}$ satisfy \eqref{c} and \eqref{C}, $\#\{\omega\in\Omega_k\mid d(y,x_\omega)\le\alpha_4r^k+1\}<\infty.$ This with (1) shows $\#\{\omega\in\Omega_k\mid B(y,1)\cap K_\omega\ne\emptyset\}<\infty$ and so $y\in\Cup_{\omega\in\Omega_k}K_\omega.$ Therefore $\cup_{\omega\in\Omega_k}K_\omega = \ol{\Cup_{\omega\in\Omega_k}K_\omega} \supset \ol{\Cup_{l\ge k}\cup_{\omega\in\Omega_l}x_\omega}=X  $ by \eqref{C} and \eqref{T1}.
		\item Let $y\in B(x_{k,n},\alpha_5r^k).$ By (1) and (2), there exists $(l,m)\in\Omega$ such that $l>k$ and $y\in K_{l,m}\subset B(x_{k,n}, \alpha_5r^k).$ Then
		\[d(x_{\pi^{l-k-1}(l,m)},x_{n,k})\le d(x_{n,k},x_{l,m})+d(x_{l,m},x_{\pi^{l-k-1}(l,m)})<(\alpha_5+\frac{r\alpha_1}{1-r})r^k\]
		similarly to (1). Therefore $\pi^{l-k}(l,m)=(n,k)$ by \eqref{T3} and $y\in  K_{k,n}.$
	\end{enumerate}
\end{proof}
\begin{lem}\label{int}
	Let $O_{k,n}=X\setminus\cup_{o\ne n}K_{k,o},$ then $O_{k,n}=\nai K_{k,n}$ and $\ol{O_{k,n}}=K_{k,n}.$
\end{lem}
\begin{proof}
	Note that $O_{k,n}=K_{k,n}\setminus\cup_{o\ne n}K_{k,o}$ by Lemma \ref{hage} (3). We first show that
	\begin{equation}
		x_{l,m}\in K_{k,n} \text{ for some } m\in\Mb{N} \text{ and }l\ge k, \text{ then }\pi^{l-k}(l,m)=(k,n).\label{ptin}
	\end{equation}
Indeed, we can take $(l',m')$ for sufficiently large $l'$ with $d(x_{l,m},x_{l',m'})<\alpha_5r^l$ and $\pi^{l'-k}(l',m')=(k,n).$ Therefore in the similar way to Lemma \ref{hage} (3), we obtain $\pi^{l'-l}(l',m')=(l,m)$ and $\pi^{l-k}(l,m)=(k,n).$\par
This shows $\Cup_{l\ge0}\{x_\lambda\mid\lambda\in\pi^{-l}(\omega)\}\subset O_\omega$ and so $\ol{O_\omega}=K_\omega$ by definition. Next we show $\nai K_\omega=O_\omega.$ Let $y\in\nai K_{k,n}\cap K_{k,o}.$ By definition of $K_{k,o},$ there exists $(l,m)$ such that $\pi^{l-k}(l,m)=(k,o)$ and $x_{l,m}\in K_{k,n}.$ This with \eqref{ptin} shows $n=o$ and $\nai K_{k,n}\subset O_{k,n}.$ On the other hand, by Lemma \ref{hage} (2) and $\#\{o\mid B(x_{k,n},\alpha_4r^k+1)\cap B(x_{k,o},\alpha_4r^k)\}<\infty,$ it follows that $O_{k,n}=$ $ \{y\in X\mid d(y, K_{k,n})<1\}\setminus\Cup_{o\ne n}K_{k,o}$ is open, and so $O_{k,n}\subset\nai K_{k,n}.$
\end{proof}
\begin{lem}\label{kiss}
	If $B(x_{k,n},\alpha_3r^k)\cap B(x_{k,o},\alpha_3r^k)\ne\emptyset,$ then $K_{k,n}\cap K_{k,o}\ne\emptyset.$
\end{lem}
\begin{proof}
	By \eqref{T4}, there exist $n_l, o_l$ for $l\ge k$ such that $\pi^{l-k}(l,n_l) = (k,n),$ $\pi^{l-k}(l,o_l) = (k,o)$ and $d(x_{l,n_l},x_{l,o_l})<2\alpha_3r^l.$ Since both $K_{k,n}$ and $K_{k,o}$ are complete, and both $\{x_{l,n_l}\}_{l\ge k}$ and $\{x_{l,o_l}\}_{l\ge k}$ are Cauchy sequences, these sequences converge to the same limit on $K_{k,o}\cap K_{k,n}.$
\end{proof}
\begin{proof}[Proof of Proposition \ref{TtoM}]
We inductively define $Q_{k,n}.$
\begin{itemize}
	\item For $(0,n)\in\Omega_0,$ we define $Q_{0,1}=K_{0,1}$ and $Q_{0,n}=K_{0,n}\setminus\Cup_{i=1}^{n-1}Q_{0,i}.$ 
	\item Let $(k,n)\in\Cup_{l>0}\Omega_l.$ If $n=\min\{o\mid\pi(k,n)=\pi(k,o)\}$, then we define $Q_{k,n}=(Q_{\pi(k,n)}\cap K_{k,n}).$ Otherwise,
	\[Q_{k,n}:=
		(Q_{\pi(k,n)}\cap K_{k,n})\setminus\cup \{Q_{k,o}\mid o<n \text{ and }\pi(k,o)=\pi(k,n) \}. \]
	\item For $(k,n)\in\Cup_{l<0}\Omega_l,$ let $Q_{k,n}=\Cup_{\omega\in\pi^k(k,n)}Q_\omega.$
\end{itemize}
Since $x_{k,n}\in O_{k,n}$ and $x_{k,n}\in K_{k+1,m}$ for some $m,$ $x_{k,n}\in\Cup_{\omega\in\pi^{-1}(k,n)}K_\omega.$ This with $\#(\pi^{-1}(k,w))<\infty$ shows $K_{k,n}=\Cup_{\omega\in\pi^{-1}(k,n)}K_\omega.$  Therefore $\eqref{q1}$ and $\eqref{q3}$ follow. Moreover,
\begin{itemize}
	\item since $O_\omega\subset O_{\pi(\omega)}$ by $K_\omega=\Cup_{\lambda\in\pi^{-1}(\omega)}K_\lambda,$ it inductively follows that for any $k\ge0$ and $n\in\Mb{N},$ it holds that $O_{k,n}\subset Q_{k,n}\subset K_{k,n}.$
	\item Let $y\in O_{k,n}$ for some $(k,n)\in\Cup_{l<0}\Omega_l.$ By definition of $Q_{k,n},$ there exists $o\in\Mb{N}$ such that $ y\in Q_{0,o}\subset K_{0,o}.$ This with $y\in O_{k,n}$ assures $\pi^{-k}(0,o)=(k,n)$ and $O_{k,n}\subset Q_{k,n}\subset K_{k,n}.$
\end{itemize}
Therefore Lemma \ref{int} leads to \eqref{q2}. The condition \eqref{q4} immediately follows from Lemma \ref{hage}. \\
Let $y,z\in X$ with $d(y,z)<(\alpha_3-C_*-\alpha_4r_0)r^k.$ By \eqref{C} and \eqref{q1}, there exist $n_0, n_1, n_2$ such that $d(y, x_{k,n_1})<C_*r^k,$ $y\in K_{k,n_0},$ and $z\in K_{k,n_2}.$ Then,
\begin{itemize}
	\item If $d(y, x_{k,n_2})<\alpha_3r^k,$ then $K_{k,n_1}\cap K_{k,n_2}\ne\emptyset$ by Lemma \ref{kiss}.
	\item Otherwise, for $(k+1,m_1)\in\pi^{-1}(k,n_2)$ with $z\in K_{k+1,m_1},$
	\begin{align*}
		d(x_{k,n_2},x_{k+1,m_1})&\ge d(y,x_{k,n_2})-d(y,z)-d(z,x_{k+1,m_1})\\
		&> (\alpha_3-(\alpha_3-C_*-\alpha_4r_0)-\alpha_4r)r^k\ge C_*r^k\\
\shortintertext{and}
	d(x_{k,n_1},x_{k+1,m_1})&\le d(y,x_{k,n_1})+d(y,z)+d(z,x_{k+1,m_1})\\
	&< (C_*+(\alpha_3-C_*-\alpha_4r_0)+\alpha_4r)r^k\le \alpha_3r^k.
\end{align*}
Therefore by \eqref{T5}, there exist $m_2,m_3$ such that $\pi(k+1,m_2)=(k,n_2)$, $\pi(k+1,m_3)=(k,n_1)$ and $B(x_{k+1,m_2},\alpha_3r^k)\cap B(x_{k+1,m_3},\alpha_3r^k)\ne\emptyset.$ This shows $K_{k,n_1}\cap K_{k,n_2} \supset K_{k+1,m_3}\cap K_{k+1,m_2} \ne\emptyset. $
\end{itemize}
Similarly, if $d(y, x_{k,n_0})\ge \alpha_3 r^k,$ then there exists $m_4$ such that $\pi(k+1,m_4)=(k,n_0),$ $y\in K_{k+1,m_4},$ $d(x_{k,n_0},x_{k+1,m_4})>C_*r^k$ and $d(x_{k,n_1},x_{k+1,m_4})<\alpha_3r^k.$ Therefore $K_{k,n_0}\cap K_{k,n_1}\ne\emptyset$ regardless of whether $d(y, x_{k,n_0})< \alpha_3 r^k$ or not, and \eqref{q5} follows.

\end{proof}
We next prove that the assumptions of Proposition \ref{TtoM} actually hold under the assumption of Theorem \ref{Main}. Let $\gamma>1$ be the constant in Definition \ref{up}, in other words, for any $y\in X$ and $r>0$ with $B(y,r)\ne X$ satisfies $B(x,\gamma r)\setminus B(x,r)\ne\emptyset.$ 

\begin{prop}\label{Tree}
	For any  $\alpha_1,\alpha_3$ with
	\begin{gather}
		\alpha_1>\alpha_3>(C_*\vee1)\gamma,\label{TC1}\\
		\intertext{let $\alpha_2>0$ be sufficiently small such that}
		(1+\gamma)\alpha_2<(\alpha_1-\alpha_3)\wedge c_* \label{TC2}\\
		\intertext{and $r_0\in(0,1)$ be also sufficiently small such that \eqref{cond1} and the following hold:}
		r_0C_*<\alpha_2, \label{TC3}\\
		%\frac{r_0}{1-r_0}\alpha_1<\alpha_2\wedge(\alpha_3-C), \label{TC4}\\
		(1+\gamma)(\alpha_2+\alpha_6r_0+C_*r_0)<c_*, \label{TC4}\\
		(1+\gamma)\alpha_2+(4+\gamma)\alpha_6r_0+(2+\gamma)C_*r_0<\alpha_1-\alpha_3\label{TC5}\\
		\shortintertext{where}
		\begin{aligned}
			N=\sup_{y\in X, s>0}\sup\{\# A\mid A\subset B(y,\alpha_1s),&\ d(z_1,z_2)\ge c_*s\\ &\text{ for any }z_1,z_2\in A\text{ with }z_1\ne z_2\},
		\end{aligned}\notag \\
		\shortintertext{and}
		\alpha_6=C_*\frac{(1+\gamma)(2+\gamma)}{\gamma-1}\bigl(\gamma^{\frac{N(N-1)}{2}}-1\bigr).\notag
	\end{gather}
Then, for any $r\in(0,r_0]$ and $A=\cup_{k\in\Mb{Z}}\{x_\omega\}_{\omega\in\Omega_k}$ with \eqref{c} and \eqref{C}, there exists a map $\pi=\pi_A: \Omega\to\Omega$ satisfying conditions \eqref{T1} to \eqref{T5}.
\end{prop}
\begin{rem}
	\begin{enumerate}
		\item $N<\infty$ because $(X,d)$ is doubling. 
		\item If $r_0$ satisfies inequalities \eqref{cond1}, \eqref{TC3}, \eqref{TC4} and \eqref{TC5}, then any $r\in(0,r_0]$ satisfies the same inequalities but replacing $r_0$ with $r$.
	\end{enumerate}
\end{rem}
\begin{proof}
	We first divide each $\Omega_{k+1}$ into three parts. We inductively define $f_{k+1}:\Mb{N}\to\Mb{N}\cup\{\infty\}$ by 
	\begin{align*}
	f_{k+1}(1)=&\min\{m\mid
	x_{k+1,m}\not\in\Cup_{\omega\in\Omega_k}B(x_\omega,(\alpha_2+\alpha_6r)r^k)\},\\
	f_{k+1}(i)=&\min\{m\mid
	x_{k+1,m}\not\in\Cup_{\omega\in\Omega_k}B(x_\omega,(\alpha_2+\alpha_6r)r^k) \text{ and} \\
	&\quad d(x_{k+1,m},x_{k+1,f_{k+1}(j)})\ge2\alpha_6r^{k+1}\text{ for any }j<i\text{ with }f_{k+1}(j)<\infty\}
	\end{align*}
 and let 
\begin{align*}
	\Omega^A_{k+1}&=\{\omega\in\Omega_{k+1}\mid x_\omega\in B(x_\lambda,\alpha_2r^k)\text{ for some } \lambda\in\Omega_k\},\\
	\Omega^B_{k+1}&=\{\omega\in\Omega_{k+1}\mid x_\omega\in B(x_{k+1,f_{k+1}(i)},\alpha_6r^{k+1})\text{ for some }i\text{ with }f_{k+1}(i)<\infty\},\\
	\Omega^C_{k+1}&=\Omega_{k+1}\setminus\bigl(\Omega^A_{k+1}\cup\Omega^B_{k+1}\bigr).
\end{align*}
Note that by definition of $f_{k+1},\ \Omega^A_{k+1}\cap\Omega^B_{k+1}=\emptyset.$ Now we define $\pi|_{\Omega_{k+1}}$ on each part.
\begin{itemize}
	\item For any $\omega\in\Omega^A_{k+1},$ there exists unique $\lambda\in\Omega_k$ such that $d(x_\omega,x_\lambda)<\alpha_2r^k$ due to \eqref{c} and $2\alpha_2<c_*$ by \eqref{TC2}. We define $\pi(\omega)=\lambda.$
	\item For any $\omega\in\Omega^C_{k+1},$ let
	\begin{equation}
		\ul{n}_\omega=\min\{n\mid d(x_\omega,x_{k,n})<C_*r^k\} \label{uln}
	\end{equation}
then $\ul{n}_\omega<\infty$ due to \eqref{C}. We set $\pi(\omega)=(k,\ul{n}_\omega).$
\item Let $i\in\Mb{N}$ with $f_{k+1}(i)<\infty,$ then by definition of $f_{k+1},$
\begin{equation*}
	\Omega^{B,i}_{k+1}:=\{\omega\in\Omega_{k+1}\mid x_\omega\in B(x_{k+1,f_{k+1}(i)},\alpha_6r^{k+1})\}
\end{equation*}
is disjoint. Let
\begin{align*}
	\Mc{N}^i_{k+1}&=\{n\in\Mb{N}\mid d(x_{k+1,f_{k+1}(i)},x_{k,n})<(\alpha_1-\alpha_6r)r^k\},\\
	P^i_{k+1}&=\{(p,q)\mid p,q\in\Mc{N}^i_{k+1}\text{ and }p<q\},
\end{align*}
then $\# P^i_{k+1}\le N(N-1)/2.$ If $P^i_{k+1}\ne\emptyset,$ we give an arbitrary order and write $P^i_{k+1}=\{(n^{i,2j-1}_{k+1},n^{i,2j}_{k+1})\mid 1\le j\le \# P^i_{k+1} \}.$ We also write $\beta_j$ a sequence defined by the recurrence relation $\beta_0=0$ and $\beta_j=\gamma\beta_{j-1}+(1+\gamma)(2+\gamma)C_*.$ Note that $\beta_j=C_*(1+\gamma)(2+\gamma)(\gamma^j-1)/(\gamma-1).$ Since $\Omega^A_{k+1}\ne\emptyset$ by \eqref{C} and \eqref{TC3}, $B(x_{k+1,f_{k+1}(i)},\beta_jr^{k+1})\ne X$ for any $j\le N(N-1)/2.$\\
Let $1\le j\le \#P^i_{k+1},$ then there exist $y,z\in X$ and $a,b\in\Mb{N}$ such that
\begin{align*}
	(\beta_{j-1}+(2+\gamma)C_*)r^{k+1}\le & d(y,x_{k+1,f_{k+1}(i)})<\gamma(\beta_{j-1}+(2+\gamma)C_*)r^{k+1},\\
	&d(y,x_{k+1,a})<C_*r^{k+1},\\
	C_*r^{k+1}\le&d(z,x_{k+1,a})<C_*\gamma r^{k+1},\\
	&d(z,x_{k+1,b})<C_*r^{k+1}
\end{align*}
because $(X,d)$ is uniformly perfect and $\eqref{C}$ holds. Note that 
\begin{equation}\label{cupling}
	z\in B(x_{k+1,a},C_*\gamma r^{k+1})\cap B(x_{k+1,b},C_*\gamma r^{k+1}).
\end{equation} We set $m^{i,2j-1}_{k+1}=a$ and $m^{i,2j}_{k+1}=b$.\par
Now we define $\pi$ on $\Omega^{B,i}_{k+1}$ (including cases of $P^i_{k+1}=\emptyset$) by
\begin{equation*}
	\pi(k+1,m)=\begin{cases}
		(k,n^{i,p}_{k+1}) &\text{ if }m=m^{i,p}_{k+1}\text{ for some }p\le2\# P^i_{k+1},\\
		(k,\ul{n}_{k+1,m}) &\text{ otherwise}
	\end{cases}
\end{equation*}
where $\ul{n}_{k+1,m}$ is as in \eqref{uln}. Then $\pi$ is well-defined because
\begin{align}
	&d(x_{k+1,m^{i,2j-1}_{k+1}},x_{k+1,m^{i,2j}_{k+1}})>C_*r^{k+1}-C_*r^{k+1}=0,\notag\\
	&d(x_{k+1,f_{k+1}(i)},x_{k+1,m^{i,2j-1}_{k+1}})\vee d(x_{k+1,f_{k+1}(i)},x_{k+1,m^{i,2j}_{k+1}})\notag\\
	<&(\gamma(\beta_{j-1}+(2+\gamma)C_*)+(2+\gamma)C_*)r^{k+1}=\beta_jr^{k+1},\label{ub1}\\
	&d(x_{k+1,f_{k+1}(i)},x_{k+1,m^{i,2j-1}_{k+1}})\wedge d(x_{k+1,f_{k+1}(i)},x_{k+1,m^{i,2j}_{k+1}})\notag\\
	>&(\beta_{j-1}+(2+\gamma)C_*-(2+\gamma)C_*)r^{k+1}=\beta_{j-1}r^{k+1}\notag
\end{align}
hold for any $1\le j\le\# P^i_{k+1},$ and so $m^{i,p}_{k+1}\ne m^{i,q}_{k+1}$ for any $p,q\le 2\# P^{i}_{k+1}$ with $p\ne q.$ Moreover, \eqref{ub1} also shows $(k+1,m^{i,p}_{k+1})\in\Omega^{B,i}_{k+1}$ for any $p.$
\end{itemize}
We next show $\pi:\Omega\to\Omega$ satisfies conditions \eqref{T1} to \eqref{T5}.
\begin{itemize}
\item[\eqref{T2}] Let $(k+1,m)\in\Omega_{k+1}.$ If $m=m^{i,p}_{k+1}$ for some $i$ and $p,$ then
\begin{align*}
	d(x_{k+1,m},x_{\pi(k+1,m)})&\le d(x_{k+1,m},x_{k+1,f_{k+1}(i)})+d(x_{k+1,f_{k+1}(i)},x_{k,n^{i,p}_{k+1}})\\
	&<(\alpha_6r+(\alpha_1-\alpha_6r))r^k=\alpha_1r^k.
\end{align*}
Otherwise, \eqref{T2} follows from $\alpha_1>C_*.$
\item[\eqref{T3}] It is obvious.
\item[\eqref{T1}]  Let $\omega\in\Omega_k,$ then there exists $\lambda\in\Omega_{k+1}$ with $d(x_\omega,x_\lambda)<C_*r^{k+1}$ by \eqref{C}.\\ Since \eqref{T3} and \eqref{TC3} hold, $\pi(\lambda)=\omega$ and so $\pi(\Omega_{k+1})=\pi(\Omega).$\\
$\sup_{\omega\in\Omega}\#\pi^{-1}(\omega)<\infty$ follows from \eqref{c} and \eqref{T2} because $(X,d)$ is doubling.
\item[\eqref{T4}] Assume $y\in B(x_{k,n_1},\alpha_3r^k)\cap B(x_{k,n_2},\alpha_3r^k)$ for some $y\in X$ with $n_1<n_2.$
\begin{itemize}
	\item  If $y\not\in\Cup_{\omega\in\Omega_k}\ol{B(x_\omega,(\alpha_2+\alpha_6r+C_*r)r^k)},$ set $z=y.$
	\item Otherwise, let $n_3=\min\{n\mid y\in B(x_{k,n},(\alpha_2+\alpha_6r+C_*r)r^k)\}.$ Since $\{x_{k,n_1},x_{k,n_2}\}\not\subset B(x_{k,n_3},(\alpha_2+\alpha_6r+C*r)r^k)$ by $2(\alpha_2+\alpha_6r+C_*r)<c_*,$ there exists $z\in X$ with \[(\alpha_2+\alpha_6r+C_*r)r^k\le d(z,x_{k,n_3})<\gamma(\alpha_2+\alpha_6r+C_*r)r^k.\]
	Then for any $l\ne n_3,$
	\begin{align*}
		d(z,x_{k,l})&\ge d(x_{k,l},x_{k,n_3})-d(x_{k,n_3},z)\\
		&\ge(c_*-\gamma(\alpha_2+\alpha_6r+C_*r))r^k>(\alpha_2+\alpha_6r+C_*r)r^k.
	\end{align*}
\end{itemize}
In both cases, $z$ satisfies $d(x_\omega,z)>(\alpha_2+\alpha_6r+C_*r)r^k$ for any $\omega\in\Omega_k.$ Let $m_0=\min\{m\mid d(x_{k+1,m},z)<C_*r^{k+1}\},$ then
$(\alpha_2+\alpha_6r)r^k<d(x_\omega,x_{k+1,m_0})$ for any $\omega\in\Omega_k,$ so there exists $i\in\Mb{N}$ such that $2\alpha_6r^{k+1}> d(x_{k+1,m_0},x_{k+1,f_{k+1}(i)}).$ Moreover,
\begin{align*}
	&d(x_{k,n_l},x_{k+1,f_{k+1}(i)})\\
	\le&d(x_{k,n_l},y)+d(y,z)+d(z,x_{k+1,m_0})+d(x_{k+1,m_0},x_{k+1,f_{k+1}(i)})\\
	<&(\alpha_3+(1+\gamma)(\alpha_2+\alpha_6r+C_*r)+C_*r+2\alpha_6r)r^k<(\alpha_1-\alpha_6r)r^k
\end{align*} 
for $l=1,2,$ by \eqref{TC5}, therefore $(n_1,n_2)\in P^i_{k+1}.$ By definition of $\pi$ on $\Omega^{B,i}_{k+1},$ there exists $j\le\# P^{i}_{k+1}$ such that $\pi(k+1,m^{i,2j-1}_{k+1})=(k,n_1)$ and $\pi(k+1,m^{i,2j}_{k+1})=(k,n_2).$ Then
\[B(x_{k+1,m^{i,2j-1}_{k+1}},\alpha_3r^{k+1})\cap B(x_{k+1,m^{i,2j}_{k+1}},\alpha_3r^{k+1})\ne\emptyset\] because \eqref{cupling} and $C_*\gamma<\alpha_3$ hold.
\item[\eqref{T5}] If $d(x_{k+1,m_1},x_{\pi(k+1,m_1)})\ge C_*r^k,$ then there exist $i,p$ with $m_1=m^{i,p}_{k+1}$ by definition of $\pi.$ Additionally, if $d(x_{k+1,m_1},x_{k,n})<\alpha_3r^k,$ then \[d(x_{k+1,f_{k+1}(i)},x_{k,n})<(\alpha_3+\alpha_6r)r^k<(\alpha_1-\alpha_6r)r^k.\]
Therefore $(n^{i,p}_{k+1},n)\in P^i_{k+1}$ or $ (n,n^{i,p}_{k+1})\in P^i_{k+1}$ holds if $n\ne n^{i,p}_{k+1},$ and so we obtain desired $m_2$ and $m_3$ by definition of $\pi$ on $\Omega^{B,i}_{k+1},$ in the same way as the proof of \eqref{T4}.
\end{itemize}
\end{proof}
Theorem \ref{Main} immediately follows from Propositions \ref{TtoM} and \ref{Tree}. 
\section{Application for evaluations of the Ahlfors regular conformal dimension}
Here we introduce the framework and results of \cite{kig,sas} to give an application of our main theorem.
\begin{defi}
Let $T$ be a countable set and $\pi:T\to T$ be a map such that the following hold:
\begin{align}
	\bullet\ & \text{let }F_\pi=\{w\mid \pi^n(w)=w\text{ for some }n\ge1\},\text{ then }\# F_\pi\le1. \tag{H1}\label{p1} \\
	\bullet\ & \text{For any }w,v\in T,\text{ there exist }n,m\ge0\text{ such that }\pi^n(w)=\pi^m(v). \tag{H2}\label{p2}
\end{align}
Let $\phi\in F_\pi$ if $F_\pi\ne\emptyset,$ otherwise we fix any $\phi\in T.$ We call the triplet $(T,\pi,\phi)$ a tree with a reference point. 
\end{defi}
We justify this definition as follows.
\begin{lem}\begin{enumerate}
		\item Let $b(w,v)=\min\{n\ge0| \pi^n(w)=\pi^m(v)\text{ for some }m\ge0\}$ for $w,v\in T,$ then $\pi^{b(w,v)}(w)=\pi^{b(v,w)}(v).$
		\item Let $\Mc{A}=\{(w,v)\mid \pi(w)=v \text{ or }\pi(v)=w\}\setminus\{(\phi,\phi)\},$ then $(T,\Mc{A})$ is a tree.
	\end{enumerate}
\end{lem}
\begin{rem}
	Since \eqref{p2} holds, $b(w,v)<\infty$ for any $w,v\in T.$
\end{rem}
\begin{proof}
	\begin{enumerate}
		\item Assume $\pi^{b(w,v)}(w)\ne\pi^{b(v,w)}(v),$ then $\pi^{b(w,v)}(w)=\pi^{m_1}(v)$ and $\pi^{m_2}(w)=\pi^{b(v,w)}(v)$ for some $m_1>b(v,w)$ and $m_2>b(w,v).$ This shows $\pi^{b(w,v)}(w)=\pi^{m_1}(v)=\pi^{(m_1-b(v,w))+(m_2-b(w,v))}(\pi^{b(w,v)}(w))$and $\pi^{b(v,w)}(v)$\\$= \pi^{(m_2-b(w,v))+(m_1-b(v,w))}(\pi^{b(v,w)}(v)),$ which contradict \eqref{p1}.
		\item \eqref{p2} assures that $(T,\Mc{A})$ is connected. Let $(w_i)_{i=0}^n$ be a simple path from $w$ to $v,$ that is, $w_0=w,$ $w_n=v,$ $(w_{i-1},w_i)\in \Mc{A}$ for any $1\le i\le n$ and $w_i\ne w_j$ for any $i\ne j.$ 
		Since $(w_i)_{i=0}^n$ is simple, there exist $0\le i_*\le n$ such that $\pi(w_i)=w_{i+1}$ for any $i<i_*$ and $w_i=\pi(w_{i+1})$ for any $i\ge i_*.$ This shows $\pi^{i_*}(w)=\pi^{n-i_*}(v),$ and it also follows that  $w_{b(w,v)}=w_{n-b(v,w)}$ because of (1). Therefore $i_*=b(w,v)$ and $n=b(w,v)+b(v,w)$ by simplicity, so the simple path from $w$ to $v$ is unique.
	\end{enumerate}
\end{proof}
\begin{rem}
	If $F_\pi\ne\emptyset,$ $(T,\Mc{A},\phi)$ coincides with ``tree with a reference point'' in the sense of \cite[Definition 2.1.2]{kig}.
	%sasのほうも 
\end{rem}
Throughout this section, $\Mc{T}=(T,\pi,\phi)$ is a tree with a reference point. We also let $[w]=b(w,\phi)-b(\phi,w)$ for $w\in T$ and $(T)_k=\{w\in T\mid [w]=k\}$ for $k\in\Mb{Z}.$
\begin{defi}[Partition]
Let $(Y,\rho)$ be a ($\sigma$-compact) metric space without isolated points, and $\Mc{C}(Y,\rho)$ be a set of all nonempty compact subsets of $(Y,\rho)$ except single points. We say $K:T\to \Mc{C}(Y,\rho)$ is a partition of $Y$ (parametrized by $\Mc{T},$) if the following conditions hold.
\begin{align}
	\bullet\ &\Cup_{w\in(T)_0}K(w)=Y\text{ and for any }w\in T,\ \Cup_{v\in\pi^{-1}(w)}K(v)=K(w). \label{P1}\tag{P1}\\
	\bullet\ &\text{For any sequence }(w_k)_{k\in\Mb{Z}}\text{ with }\pi(w_{k+1})=w_k \text{ for any }k\in\Mb{Z},\notag\\
	& \cap_{k\in\Mb{Z}}K(w_k)\text{ is a single point.} \label{P2}\tag{P2}
\end{align}
\end{defi} 
Hereafter, we write $K_w$ instead of $K(w)$ for simplicity. In \cite[Definition 2.2.1]{kig}, the following condition is also included in the definition of partition.
\begin{equation}
 \bullet\ \text{For any }w\in T,\ K_w\text{ has no isolated points.} \tag{P*}\label{Pstar}
\end{equation}
\begin{lem}
	\eqref{Pstar} follows for any partition $K$.
\end{lem}
\begin{proof}
	Assume $x\in K_w$ be an isolated point of $K_w$ for some $w\in T.$ By \eqref{P1}, we can find $(w_k)_{k\in\Mb{Z}}$ with $\pi(w_{k+1})=w_k$ and $x\in K_{w_k}$ for any $k\in\Mb{Z},$ and $w_0=w.$ Since $x$ is an isolated point of $K_{w_k}\subset K_{w_0}$ for any $k\ge0$ and $K_{w_k}$ is not a single point, $(K_{w_kl}\setminus\{x\})_{k\ge0}$ is a decreasing sequence of nonempty compact sets. Therefore $\Cap_{k\in\Mb{Z}}K_{w_k}\setminus\{x\}=\Cap_{k\ge0}(K_{w_k}\setminus\{x\})\ne\emptyset,$ this contradicts \eqref{P2}.
\end{proof}
\begin{defi}
	For $w\in T$ and $s\in(0,\infty),\ x,y\in X$ and $M\ge1,$ we define
\begin{align*}
	g_\rho(w)&=\diam (K_w,\rho), \\
	\Lambda^\rho_s&=\begin{cases}
		\emptyset& \text{if }F_\pi\ne\emptyset\text{ and }s>g_\rho(\phi)\\
		\{w\in T\mid g_\delta(w)\le s<g_\rho(\pi(w))\} & \text{otherwise,}
	\end{cases}\\
E^\rho_s&=\{(w,v)\in\Lambda^\rho_s\times\Lambda^\rho_s\mid w\ne v\text{ and }K_w\cap K_v\ne\emptyset\},\\
\shortintertext{$l^\rho_s(\cdot,\cdot)$ is the graph distance of $(\Lambda^\rho_s,E^\rho_s)$ and}
\delta^\rho_M(x,y)&=\inf\{s>0\mid\text{there exist }w,v\in \Lambda^\rho_s\text{ such that }\\
&\hspace{120pt} x\in K_w,\ y\in K_v\text{ and }l^\rho_s(w,v)\le M\}.
\end{align*}
\end{defi}
\begin{defi}[Basic framework]
	Assume $\sup_{w\in T}\#\pi^{-1}(w)<\infty$ and let $K$ be a partition of $(Y,\rho)$ such that for any $w\in T,$ $K_w\setminus \Cup_{v\in (T)_{[w]}:v\ne w}K_v\ne\emptyset,$ and there exists an open set $U_w\supset K_w$ with $\#\{v\in (T)_{[w]}\mid U_w\cap K_v\ne\emptyset \}<\infty.$ We say $\rho$ satisfies basic framework with respect to $K$ if the following conditions hold.
	\begin{align}
	\bullet\ &\text{(Adapted).}\  
	\text{There exists }M\ge1,\ \eta_1>0 \text{ such that}\notag\\
	&\hspace{75pt}\eta_1^{-1}\delta^\rho_M(x,y)\le\rho(x,y)\le\eta_1\delta^\rho_M(x,y)\text{ for any }x,y\in Y. \tag{B1}\label{B1}\\
	\bullet\ &\text{(Thick).}\ \text{There exists }\eta_2>0\text{ such that for any }w\in T,\notag\\
	&\hspace{60pt}\{y\mid \delta^\rho_1(x_w,y)\le\eta_2g_\rho(\pi(w))\}\subset K_w\text{ with some }x_w\in K_w. \tag{B2}\label{B2}\\
	\bullet\ &\text{(Uniformly finite).}\ sup_{s\in(0,\infty),\ w\in\Lambda^\rho_s}\#\{v\mid v\in\Lambda^\rho_s,\ l^\rho_s(w,v)\le 1\}<\infty. \tag{B3}\label{B3}\\
	\bullet \ &\text{There exist }\eta_3>0\text{ and }r\in(0,1)\text{ such that }\notag\\ 
	&\hspace{110pt}\eta_3^{-1}r^{[w]}\le g_\rho(w)\le \eta_3r^{[w]}\text{ for any }w\in T.\tag{B4}\label{B4}
	\end{align}
\end{defi}
The main result of \cite{kig} is the following.
\begin{thm}[\cite{kig},Theorem 4.6.4, \cite{sas}, Theorem 3.9]\label{thar}
	Let $K$ be a partition of $(Y,\rho),\ E_k=\{(w,v)\in (T)_k\times (T)_k\mid w\ne v\text{ and }K_w\cap K_v\ne\emptyset \},\ l_k$ be the graph distance of $((T)_k,E_k)$ and 
\begin{multline*}
	\Mc{E}_{p,k,w,M}=\inf\{\frac{1}{2}\sum_{(u,v)\in E_{[w]+k}}|f(u)-f(v)|^p\mid
	f:(T)_{[w]+k}\to\Mb{R},\ f(u)=1\\
	\text{ if }\pi^k(u)=w\text{ and }f(u)=0\text{ if }l_{[w]}(w,\pi^k(u))>m\}.
\end{multline*}
	If $\rho$ satisfies basic framework with respect to $K,$ then
	\[\ard(Y,\rho)=\inf\{p\mid\limsup_{k\to\infty}\sup_{w\in T}\Mc{E}_{p,k,w,M}=0\},\]
	where $M$ is the constant appearing in \eqref{B1}.
\end{thm}
We can adapt our main theorem to this framework. Recall that $(X,d)$ is a complete, doubling, uniformly perfect metric space.
\begin{prop}\label{part}
	There exist a tree with a reference point $(T,\pi,\phi)$ and a partition $K$ of $(X,d)$ such that $d$ satisfies the basic framework with respect to $K.$
\end{prop}
\begin{proof}
	Fix any $x_*\in X.$ Then by Lemma \ref{net}, we can take $\Cup_{k\in\Mb{Z}} \{x_\omega\}_{\omega\in\Omega_k}$ with \eqref{c}, \eqref{C} and $x_*\in\{x_\omega\}_{\omega\in\Omega_k}$ for any $k\in\Mb{Z}.$ Let $Q_\omega\subset X\ (\omega\in\Omega)$ be sets given by Theorem \ref{Main}. We also let $k_0=\max\{k\in\Mb{Z}\mid \#\Omega_k=1\}$ with $\max\emptyset=-\infty$ and $T=\cup_{k\ge k_0}\Omega_k.$ If $k_0>-\infty$ and $\omega\in\Omega_{k_0},$ we define $\pi(\omega)=\omega.$ Otherwise, let $\pi(\omega)$ be the unique vertex in $\Omega_{k-1}$ such that $Q_\omega\subset Q_{\pi(\omega)}:$ existence and uniqueness follow from \eqref{q1} and \eqref{q3}. Then $\pi$ satisfies the hypothesis \eqref{p1} and \eqref{p2} by \eqref{q4} and $x_*\in\cap_{k\in\Mb{Z}}\{x_\omega\mid\omega\in\Omega_k\}.$\par
	Let $K_\omega=\ol{Q_\omega}.$ Since every bounded sets on a doubling metric space is totally bounded, \eqref{q1}, \eqref{q3} and \eqref{q4} with the uniformly perfect condition of $(X,d)$ show that $K$ is a partition of $(X,d).$ \par
	\eqref{B4} also follows from these conditions. \eqref{B3}, $\sup_{w\in T}\#\pi^{-1}(w)<\infty,$ and $\#\{v\in (T)_{[w]}\mid U_w\cap K_v\ne\emptyset \}<\infty$ for some $U_w\supset K_w$ follow from \eqref{q4} and the doubling condition. By \eqref{q5}, $\delta^d_2(x,y)\le \eta_1d(x,y)$ for some $\eta_1>0,$ which imply \eqref{B1} because of \cite[Theorem 2.4.5]{kig}. \eqref{B1}, \eqref{B3} and \eqref{q3} also imply \eqref{B2}. Finally, $K_w\setminus \Cup_{v\in (T)_{[w]}:v\ne w}K_v\ne\emptyset$ follows from \eqref{q2} because of \cite[Proposition2.2.3]{kig} and \cite[Lemma 3.10]{sas}.
\end{proof}
\begin{thm}
	Let $(Y,\rho)$ be a complete metric space without isolated points. Then the following conditions are equivalent.
	\begin{enumerate}
		\item $(Y,\rho)$ is doubling and uniformly perfect.
		\item $\ard(Y,\rho)<\infty.$
		\item There exist a tree with a reference point and a partition $K$ of $(Y,\rho)$ such that $\rho$ satisfies the basic framework with respect to $K.$
	\end{enumerate}
\end{thm}
\begin{proof}
	$((1)\Leftrightarrow(2))$ This follows from \cite[Theorem 13.3 and Corollary 14.15]{Hei}.\par
	$((1)\Rightarrow(3))$ It is shown in Proposition \ref{part}.\par
	$((3)\Rightarrow(1))$ By \cite[Proposition 4.3.1]{kig}, we may assume $r<\eta_3^{-2}$ for constants in the basic framework, and so we may take $\Lambda^\rho_{\eta_3r^k}=(T)_k$ without loss of generality. If $A\subset B(x,2r)$ satisfies $d(y,z)\ge r$ for any $y,z\in A$ with $w\ne z,$ \eqref{B1} and \eqref{B4} imply that there exists $k\in\Mb{N}$ independent of $x$ and $r,$ and there exist $\omega$ and $\lambda_y\in Y(y\in A)$ such that $\lambda_y\in (T)_{[\omega]+k},\ l_{[\omega]}(\omega,\pi^k(\lambda_y))\le M$ and $\lambda_y\ne\lambda_z$ if $y\ne z.$ Therefore $\#A\le(\sup_{\omega\in T} \#\pi^{-1}(\omega))^k (\sup_{s,\omega}\#\{v\mid v\in\Lambda^\rho_s,\ l^\rho_s(w,v)\le 1\})^M$ and so $(Y,\rho)$ is doubling. On the other hand, \cite[Lemma 3.6.4]{kig} shows $(Y,\rho)$ is uniformly perfect. 
\end{proof}
\section*{Acknowledgments}
This work was supported by JSPS KAKENHI Grant Number JP20J23120. 


\begin{thebibliography}{99}
	\bibitem{AH} 
	P. Auscher and T. Hyt\"onen, Orthonormal bases of regular wavelets in spaces of homogeneous type. \textit{Appl. Comput. Harmon. Anal.} \textbf{34} (2013), no. 2, 266–296.
	
	\bibitem{BFP} 
	F. Bernicot, D. Frey and S. Petermichl, Sharp weighted norm estimates beyond Calderón-Zygmund theory. \textit{Anal. PDE} \textbf{9} (2016), no. 5, 1079–1113.
	
	\bibitem{BK} 
	M. Bonk and B. Kleiner, Conformal dimension and Gromov hyperbolic groups with 2-sphere boundary. \textit{Geom. Topol.} \textbf{9} (2005), 219-246. %\MR{2116315}
	
	\bibitem{BP}
	M. Bourdon and H. Pajot, Cohomologie $\ell_p$ et espaces de Besov. \textit{J. Reine Angew. Math.} \textbf{558} (2003), 85-108. %\MR{1979183}
	
	\bibitem{Chr}
	M. Christ, A $T(b)$ theorem with remarks on analytic capacity and the Cauchy integral. \textit{Colloq. Math.} \textbf{60/61} (1990), 601–628.
	
	\bibitem{CJKS} 
	T. Coulhon, R. Jiang, P. Koskela and A. Sikora, Gradient estimates for heat kernels and harmonic functions. 
	\textit{J. Funct. Anal.} \textbf{278} (2020), no. 8, 108398, 67 pp.
	
	\bibitem{Da88}
	G. David, Morceaux de graphes lipschitziens et int\'egrales singuli\`eres sur une surface. \textit{Rev. Mat. Iberoamericana} \textbf{4} (1988), no. 1, 73–114. 
	
	\bibitem{Da91}
	G. David, \textit{Wavelets and singular integrals on curves and surfaces.} Lecture Notes in Mathematics, 1465. Springer-Verlag, Berlin, 1991.
	
	\bibitem{GL}
	G. Gigante and P. Leopardi, Diameter bounded equal measure partitions of Ahlfors regular metric measure spaces. \textit{Discrete Comput. Geom.} \textbf{57} (2017), no. 2, 419–430.
	\bibitem{Hei}
	J. Heinonen, \textit{Lectures on analysis on metric spaces.} Universitext. Springer-Verlag, New York, 2001.%\MR{1800917}
	
	\bibitem{HK12}
	T. Hyt\"onen and A. Kairema, Systems of dyadic cubes in a doubling metric space. \textit{Colloq. Math.} \textbf{126} (2012), no. 1, 1–33. %\MR{2901199}
	
	\bibitem{KRS}
	A. K\"aenm\"aki, T. Rajala and V. Suomala, Existence of doubling measures via generalised nested cubes. \textit{Proc. Amer. Math. Soc.} \textbf{140} (2012), no. 9, 3275–3281. 
	
	\bibitem{kig}
	J. Kigami, \textit{Geometry and analysis of metric spaces via weighted partitions.} Lecture Notes in Mathematics, 2265. Springer, Cham, 2020. %MR{4175733}
	
	\bibitem{sas}
	K. Sasaya Ahlfors Regular Conformal Dimension of Metrics on Infinite Graphs and Spectral Dimension of the Associated Random Walks.  To appear in \textit{J. Fractal Geom.} \\
	The preprint version is available in arXiv: 2009.03595 [math.PR].
	
	\bibitem{TV}
	P. Tukia and J. V\"ais\"al\"a, Quasisymmetric embeddings of metric spaces. \textit{Ann. Acad. Sci. Fenn. Ser. A I Math.} \textbf{5} (1980), no. 1, 97-114. %\MR{0595180} 
\end{thebibliography}
\end{document}